\documentclass[12pt]{article}
\usepackage{amsmath,amssymb,amsthm,tikz,fullpage}

\newcommand{\cart}{\mathop{\Box}}
\newcommand{\ceil}[1]{\left \lceil #1 \right \rceil}
\newcommand{\floor}[1]{\left \lfloor #1 \right \rfloor}
\newcommand{\size}[1]{\left \vert #1 \right \vert}

\def\qed{\hfill\rule{2mm}{2mm}\par\bigskip}

\def \gmmax{\alpha_g'}
\def \gmmin{\hat \alpha_g'}
\def \match{\alpha'}
\def \sat{\mathrm{sat}}
\def \ex{\mathrm{ex}}
\def \gmsat{\sat_g}
\def \gmsatmin{\widehat{\sat}_g}

\newtheorem{theorem}{Theorem}[section]

\newtheorem{lemma}[theorem]{Lemma}
\newtheorem{corollary}[theorem]{Corollary}
\newtheorem{proposition}[theorem]{Proposition}

\theoremstyle{definition}
\newtheorem{definition}[theorem]{Definition}

\newtheorem{example}[theorem]{Example}

\def\Max{\alpha'_g}
\def\Min{\hat\alpha'_g}
\def\ap{\alpha'}
\def\C#1{\left|{#1}\right|}
\def\FR{\frac}
\def\FL{\floor}
\def\esub{\subseteq}

\def\Hb{\overline{H}}

\def\NN{\mathbb{N}}
\def\UE#1#2#3{\bigcup_{#1=#2}^{#3}}
\def\almin{\mu}

\begin{document}

\title{Game matching number of graphs}

\author{
Daniel W. Cranston\thanks{Department of Mathematics and Applied
Mathematics, Virginia Commonwealth University, dcranston@vcu.edu},
William B. Kinnersley\thanks{Mathematics Department,
University of Illinois, wkinners@gmail.com.
Research supported by National Science Foundation grant DMS 08-38434
``EMSW21-MCTP: Research Experience for Graduate Students''.},
Suil O\thanks{Mathematics Department, College of William and Mary,
so@wm.edu},
Douglas B. West\thanks{Mathematics Department,
Zhejian Normal University and University of Illinois, west@math.uiuc.edu.
Research supported by the National Security Agency under Award
No.~H98230-09-1-0363.}
}

\date{revised December, 2012}
\maketitle

\begin{abstract}
We study a competitive optimization version of $\alpha'(G)$, the maximum size
of a matching in a graph $G$.  Players alternate adding edges of $G$ to a
matching until it becomes a maximal matching.  One player (Max) wants the final
matching to be large; the other (Min) wants it to be small.  The resulting
sizes under optimal play when Max or Min starts are denoted $\Max(G)$ and
$\Min(G)$, respectively.  We show that always $\C{\Max(G)-\Min(G)}\le1$.  We
obtain a sufficient condition for $\Max(G)=\ap(G)$.  Always
$\Max(G)\ge \FR23\ap(G)$, with equality for many split graphs, while
$\Max(G)\ge\FR34\ap(G)$ when $G$ is a forest.  Whenever $G$ is a $3$-regular
$n$-vertex connected graph, $\Max(G) \ge n/3$, and such graphs exist with
$\Max(G)< 7n/18$.  For an $n$-vertex path or cycle, the value is roughly $n/7$.
\end{abstract}

\baselineskip 16pt

\section{Introduction}
The archetypal question in extremal graph theory is ``how many edges can an
$n$-vertex graph contain without containing a copy of a forbidden subgraph
$F$?''  The answer is the {\em extremal number} of $F$, denoted $\ex(F;n)$.  The
celebrated theorem of Tur\'an~\cite{Tur} gives a formula for $\ex(K_r;n)$ and
characterizes the largest graphs not containing $K_r$.

A graph $H$ is {\em $F$-saturated} if $F \not \subseteq H$, but $F\esub H+e$
whenever $e\in E(\Hb)$.  Thus the extremal number $\ex(F;n)$ is
the maximum size of an $F$-saturated $n$-vertex graph.  The {\em saturation
number} of $F$, denoted $\sat(F;n)$, is the minimum size of an $F$-saturated
$n$-vertex graph.  Erd\H{o}s, Hajnal, and Moon~\cite{EHM} initiated the study
of graph saturation, determining $\sat(K_r;n)$.  
\looseness-1

More generally, for fixed graphs $F$ and $G$, a subgraph $H$ of $G$ is
{\em $F$-saturated relative to $G$} if $F\not\esub H$, but $F\esub H+e$ 
whenever $e\in E(G)-E(H)$.  When edges are successively added to $H$ to reach a
subgraph that is $F$-saturated relative to $G$, the final number of edges is
between the minimum and maximum sizes of such subgraphs.  When the edges are
chosen by two players with opposing objectives, we obtain a natural game.

\begin{definition}
The {\em $F$-saturation game} on a graph $G$ is played by two players, Max and
Min.  The players jointly construct a subgraph $H$ of $G$ by alternately adding
edges of $G$ to $H$ without completing a copy of $F$.  The game ends when $H$
is $F$-saturated relative to $G$.  Max aims to maximize the length of the game,
while Min aims to minimize it.  The {\em game $F$-saturation number} of $G$ is
the length of the game under optimal play, denoted by $\gmsat(F,G)$ when Max
starts and by $\gmsatmin(F,G)$ when Min starts.
\end{definition}

The saturation game bears similarities to other games.  In a {\em Maker-Breaker}
game, the players choose edges of a graph $G$ in turn.  Maker wins by claiming
the edges in a subgraph having a desired property ${\cal P}$; Breaker wins by
preventing this.  For example, Hefetz, Krivelevich, Stojakovi\'c, and
Szab\'o~\cite{HKSS1} studied Maker-Breaker games on $K_n$ in which Maker seeks
to build non-planar graphs, non-$k$-colorable graphs, or $K_t$-minors.

In a Maker-Breaker game, one may ask how quickly Maker can win.  Breaker then
acts like Max in the saturation game, aiming to prolong the game.  Hefetz,
Krivelevich, Stojakovi\'c, and Szab\'o~\cite{HKSS2} showed that on $K_n$, Maker
can build a spanning cycle within $n+2$ turns and can build a $k$-connected
subgraph within $(1+o(1))kn/2$ turns.  Feldheim and Krivelevich~\cite{FK}
showed that Maker can build a given $d$-degenerate $p$-vertex graph within
$d^{11}2^{2d+7}p$ moves (when $n$ is large relative to $d$ and $p$).

Saturation games were introduced by F\" uredi, Reimer, and Seress \cite{FRS};
they studied the value $\gmsat(K_3, K_n)$ as ``a variant of Hajnal's
triangle-free game''.  (In Hajnal's original ``triangle-free game'', the
players try to avoid triangles, with the loser being the player forced to
create one; see~\cite{CHR,GP,MS1,MS2,Pra,Ser}.)  Since the $F$-saturation game
always produces an $F$-saturated graph, trivially
$n-1 = \sat(K_3;n) \le \gmsat(K_3,K_n) \le \ex(K_3;n) = \floor{n^2/4}$.
The lower bound on $\gmsat(K_3,K_n)$ from~\cite{FRS} is $\Omega(n\lg n)$.
An unpublished result of Erd\H{o}s states $n^2/5$ as an upper bound.  Bir\'o,
Horn, and Wildstrom~\cite{BHW} improved the leading coefficient in the upper
bound, but the correct order of growth remains unknown.

In this paper, we study the $P_3$-saturation game.  The $P_3$-saturated
subgraphs of $G$ are the maximal matchings in $G$, so we call $\gmsat(P_3,G)$
the {\em game matching number of $G$}.  With $\ap(G)$ denoting the maximum size
of a matching in $G$ (the {\it matching number}), we let $\Max(G)$ and
$\Min(G)$ denote the sizes of the matchings produced under optimal play
in the {\it Max-start game} (where Max plays first) and
in the {\it Min-start game} (where Min plays first), respectively.

The outcome of an $F$-saturation game may depend heavily on which player starts.
For example, if $G$ arises from $K_{1,k}$ by subdividing one edge, then
$\gmsat(G,2K_2) = k$, but $\gmsatmin(G,2K_2) = 2$.  The special case of game
matching ($F=P_3$) is much better behaved; the main result of
Section~\ref{sec_gamematch} states that always $\size{\gmmax(G)-\gmmin(G)}\le1$.
The proof involves also proving $\Max(H)\le\Max(G)$ and $\Min(H)\le \Min(G)$
when $H$ is an induced subgraph of $G$.

Section~\ref{sec_match_gamematch} examines the relationship between $\gmmax(G)$
and $\match(G)$.  We obtain a sufficient condition for $\gmmax(G)=\match(G)$;
it is somewhat technical but is preserved by taking the cartesian product with
any other graph.  We also prove $\gmmax(G)\ge\FR23\match(G)$ for all $G$.  This
inequality is sharp; equality holds for many split graphs (a graph is a
{\it split graph} if its vertex set can be partitioned into a clique and an
independent set).  We also show that the maximum number of edges in an
$n$-vertex graph such that $\ap(G)=3k$ and $\Max(G)=2k$ is
$\binom{3k}2+3k(n-3k)$ when $n\ge6k$.  Finally, we prove the general upper
bound $\gmmin(G)\le \FR32\mu(G)$, where $\mu(G)$ is the minimum size of a
maximal matching in $G$.

In Section~\ref{forests}, we restrict our attention to forests, where the 
lower bound in terms of $\ap(G)$ can be improved.  When $G$ is a forest,
$\gmmax(G) \ge \FR34\match(G)$, which is sharp (equality holds for the ``comb''
obtained by adding a pendant edge at each vertex of the path $P_{4k}$).
We also prove that $\Min(F)\le\Max(F)$ when $F$ is a forest, which means that
there is no advantage in starting second or skipping a turn.  A closely
related result is that adding a star component to a forest increases both
parameters by $1$.

Finally, in Section~\ref{smalldeg} we consider graphs with small maximum 
degree.  We show that optimal play on the $n$-vertex path $P_n$ produces a
maximal matching of size differing from $n/7$ by less than $3$.  Note also
that the outcome on $C_n$ is greater by $1$ than the outcome on $P_{n-2}$ with
the other player starting.  For a connected $3$-regular $n$-vertex graph $G$,
we prove that always $\Max(G)\ge n/3$, and we construct such graphs with
$\Max(G)< 7n/18$.  There are disconnected $3$-regular graphs with
$\Max(G)\le 3n/8$.

\section{Max-start vs.~Min-start}\label{sec_gamematch}

Our primary goal in this section is to determine the positive integer pairs
$(r,s)$ that are realizable as $(\Max(G),\Min(G))$ for some graph $G$.  We show
first that all pairs with $|r-s|\le 1$ are realizable (except $(1,2)$).  The
main result is then that these are the only such pairs; in other words,
the choice of the starting player makes little difference.  As part of the
proof, we will show also that vertex deletion cannot increase $\Max$ or
$\Min$.

We state explicitly the special case for game matching of the trivial upper
and lower bounds for game saturation: $\Max(G)$ and $\Min(G)$ are bounded above
by $\match(G)$ and bounded below by $\almin(G)$, the minimum size of a maximal
matching in $G$.

\begin{proposition}
A pair $(r,s)$ of positive integers is realizable as $(\Max(G),\Min(G))$ for
some graph $G$ if $|r-s|\le1$ (except for $(r,s)=(1,2)$), and $G$ may be
required to be connected.
\end{proposition}
\begin{proof}
If Min cannot move without leaving another move, then the same is true for a
first move by Max, so $(1,2)$ is not realizable.
The complete graph $K_{2r}$ realizes $(r,r)$.  To realize $(r,r-1)$, add a
pendant vertex to $K_{2r-1}$.

To realize $(r,r+1)$, we present a graph $G$ and give strategies for the second
player to ensure $\Max(G)\le r$ and $\Min(G)\ge r+1$.  These values are optimal,
because in each case we will have $\Max(G)\ge\almin(G)=r$ and
$\Min(G)\le\match(G)=r+1$.

To realize $(2k,2k+1)$, take $K_{4k+2}$ and discard a perfect matching.  When
$k=1$, the graph after the first move is always $K_{1,1,2}$; Min moving next
can prevent a third move, while Max moving next can guarantee a third move.
For $k\ge2$, any move by the first player joins endpoints of two deleted edges,
and the second player can play the edge joining the other two endpoints of
those edges to reach the same situation in a smaller graph.

To realize $(2k-1,2k)$ for $k\ge 2$, take $2K_{2k}$, delete one edge from each
component, creating four {\it special} vertices, and restore regularity by
adding a different pair of edges on the special vertices.  If Min makes a first
move involving a special vertex, then Max can move to leave $2K_{2k-2}$.
If Max makes a first move involving a special vertex, then Min can move to
leave $K_{2k-3}+K_{2k-1}$ (we use ``$+$'' to denote disjoint union).  A first
move not involving a special vertex can be mirrored on the vertices of the
other large clique to reach the same situation in a smaller graph, except that
when $k=2$ such a first move by Max can be answered by Min to leave $P_3+K_1$.
\end{proof}

The pathology $\Max(G)<\Min(G)$ cannot occur when $G$ is a forest; we prove
this in Section~\ref{forests}.  Meanwhile, we begin the proof of the general
bounds $|\Max(G)-\Min(G)|\le 1$ with a simple observation that will also be
useful in other contexts.

\begin{proposition}\label{optimal}
If $uv$ is an edge in a graph $G$, then $\gmmax(G)\ge1+\gmmin(G-\{u,v\})$, with
equality if and only if $uv$ is an optimal first move for Max on $G$.
Likewise, $\gmmin(G)\le1+\gmmax(G-\{u,v\})$, with equality if and only if $uv$
is an optimal first move for Min on $G$.
\end{proposition}
\begin{proof}
The right side of each claimed inequality is the result under optimal play
after $uv$ is played as the first move.  The first player does at least as well
as this, with equality (by definition) if and only if $uv$ is an optimal first
move.
\end{proof}

We facilitate the inductive proof of $|\Max(G)-\Min(G)|\le1$  by proving
simultaneously that both game matching numbers are monotone under the deletion
of vertices. 

\begin{theorem}\label{monotone}\label{at_most_one}
If $G$ is a graph, and $v$ is a vertex in $G$, then\\
(1) $|\Max(G)-\Min(G)|\le 1$, and\\
(2) $\Max(G)\ge \Max(G-v)$ and $\Min(G)\ge \Min(G-v)$.
\end{theorem}
\begin{proof}
We prove both statements simultaneously by induction on $|V(G)|$.  They hold
by inspection for $|V(G)|\le 2$, so consider larger $G$.

\medskip
{\bf Step 1:} {\it If (1) and (2) hold for smaller graphs, then (2) holds for
$G$.}
First consider $\Max(G)$.  Let $H=G-v$, and let $xy$ be an optimal first move
in the Max-start game on $H$.  Let $H'=H-\{x,y\}$ and $G'=G-\{x,y\}$.  By
Proposition~\ref{optimal}, $\gmmax(H)=1+\gmmin(H')$ and
$\gmmax(G)\ge1+\gmmin(G')$.  Since $H'=G'-v$, applying (2) for $G'$ yields
$$\gmmax(G) \ge 1 + \gmmin(G') \ge 1 + \gmmin(H') = \gmmax(H).$$

Now consider $\Min(G)$, with again $H=G-v$.  Let $xy$ be an optimal first move
in the Min-start game on $G$ (not on $H$!), and let $G'=G-\{x,y\}$.

If $v\notin\{x,y\}$, then let $H'=H-\{x,y\}$; here $H'=G'-v$.  If $v=x$, and
$y$ has a neighbor $z$ in $H$, then let $H'=H-\{y,z\}$; here $H'=G'-z$.  In
both cases, applying optimality of $xy$ for the first move on $G$, statement
(2) for $G'$, and Proposition~\ref{optimal} for $H$ yields
$$\gmmin(G) = 1 + \gmmax(G') \ge 1 + \gmmax(H') \ge \gmmin(H).$$

By symmetry in $x$ and $y$, the only remaining case is $v=x$ with $y$ isolated
in $H$.  Here the irrelevance of isolated vertices and (1) for $H$ yield
$$\gmmin(G)=1+\gmmax(G')=1+\gmmax(H-y)=1+\gmmax(H) \ge \gmmin(H).$$

\medskip
{\bf Step 2:} {\it If (2) holds for $G$ (and smaller graphs), then (1) holds
for $G$.}
To prove $\gmmin(G) \le 1 + \gmmax(G)$, let $uv$ be an optimal first move in
the Min-start game on $G$, and let $G' = G - \{u,v\}$.  Applying (2) to both
$G-v$ and $G$ yields
$$\gmmin(G) = 1 + \gmmax(G') \le 1+\gmmax(G-v)\le 1 + \gmmax(G).$$
The inequality $\gmmax(G) \le 1 + \gmmin(G)$ follows by the same computation
with $\Max$ and $\Min$ exchanged, where $uv$ an optimal first move in the
Max-start game on $G$.
\end{proof}

As a corollary, we obtain the following result:

\begin{corollary}\label{del_vertex}
If $v$ is a vertex in a graph $G$, then
$\Max(G)\ge \Max(G-v)\ge \Max(G)-2$ and $\Min(G)\ge \Min(G-v)\ge \Min(G)-2$,
and the bounds are sharp.
\end{corollary}
\begin{proof}
Theorem~\ref{at_most_one} yields the upper bounds for $G-v$, with equality
when $v$ is isolated.

The lower bounds hold (strictly) when $v$ is isolated, so we may assume that
$v$ has a neighbor $u$.  By Proposition~\ref{optimal} and
Theorem~\ref{at_most_one}(1),
$\gmmin(G) \le 1 + \gmmax(G - \{u,v\}) \le 2 + \gmmin(G - \{u,v\})$, and then
$\gmmin(G-v) \ge \gmmin(G - \{u,v\}) \ge \gmmin(G) - 2$ by
Theorem~\ref{monotone}(2).  For the other inequality, Proposition~\ref{optimal}
and Theorem~\ref{at_most_one} yield
$\gmmax(G-v) \ge \gmmax(G-\{u,v\}) \ge \gmmin(G) - 1 \ge \gmmax(G) - 2$. 

For sharpness of the lower bounds, let $G=rK_2+C_6$ with $r\ge1$, and let
$v$ be a vertex of an isolated edge.  Since $\Max(C_6)=2<3=\Min(C_6)$,
neither player wants to play on the $6$-cycle.  When $r$ is odd, the first
player will have to play on the $6$-cycle in $G-v$; when $r$ is even, it will
be the second player.  Hence $\Max(G-v)=r+1=\Max(G)-2$ when $r$ is odd, and
$\Min(G-v)=r+1=\Min(G)-2$ when $r$ is even.
\end{proof}

\section{Relation to Matching Number}\label{sec_match_gamematch}

We next study the relationship between the game matching number and the
ordinary matching number.  Although generally $\Max(G)<\ap(G)$, there is a
condition sufficient for equality.

\begin{theorem}\label{gamen2}
Fix an $n$-vertex graph $G$ and a maximum matching $M$ in $G$.  If
$uv \in E(G)$ implies $u'v' \in E(G)$ whenever $uu',vv' \in M$, then
$\gmmax(G) = \gmmin(G) = \ap(G)$.
\end{theorem}
\begin{proof}
Both claims hold by inspection when $n \le 4$; we proceed by induction on $n$.

Since neither $\gmmax(G)$ nor $\gmmin(G)$ can exceed $\ap(G)$, we need only
give strategies for Max.  Let $uv$ be the first edge played, and let
$G'=G-\{u,v\}$.  If $uv\in M$, then $\ap(G')=\ap(G)-1$, and $M-\{uv\}$
satisfies the hypothesis for $G'$, so the induction hypothesis applies.

Hence in the Max-start game it suffices to choose $uv\in M$, and in the
Min-start game we may assume $uv\notin M$.  If $u$ and $v$ are not both covered
by $M$, then $\ap(G')=\ap(G)-1$ (or $G$ would have a larger matching); again 
the induction hypothesis applies.

In the remaining case, there exist $uu',vv'\in M$.  By hypothesis,
$u'v' \in E(G)$; Max responds by playing $u'v'$.  Now $G-\{u,v,u',v'\}$
satisfies the hypothesis (using $M-\{uu',vv'\}$), so the induction hypothesis
yields
$
\gmmin(G) \ge 2+\gmmin(G-\{u,v,u',v'\}) = 2+\ap(G)-2.
$
\end{proof}

The property we require of $G$ in Theorem~\ref{gamen2} is restrictive, but if
$G$ has a perfect matching $M$ satisfying the hypothesis of
Theorem~\ref{gamen2}, then so does the cartesian product $G \cart H$, for any
graph $H$ (using the perfect matching in $G \cart H$ formed by the copies of
$M$).  For example:

\begin{corollary}
For $r\ge1$ and any graph $H$, Max can force a perfect matching
in $K_{r,r}\cart H$, no matter who plays first.
\end{corollary}

Theorem~\ref{gamen2} does not allow Max to force a perfect matching in all
cartesian products that have perfect matchings.

\begin{example}
Let $G$ be the ``paw'', obtained from a triangle with vertex set $u,v,w$ by
adding one vertex $x$ with neighbor $w$.  In $G\cart P_3$, where $P_3$ is the
path with vertices $a,b,c$ in order, Max cannot force a perfect matching no
matter who starts.  Min can start by playing the edge joining $(w,a)$ and
$(v,a)$, threatening to isolate $(u,a)$ or $(x,a)$.  Max cannot prevent Min
from isolating one of these vertices on the next turn.  Similar analysis
applies to the Max-start game on $G\cart P_3$ and to both games on $G\cart G$.
\qed
\end{example}

Next we consider how small $\gmmax(G)$ can be in terms of $\match(G)$; we
prove a general lower bound and show that it is sharp.  A {\it round} of play
consists of a move by Max followed by a move by Min.  We call
$G - \UE i1k \{u_i,v_i\}$ the \emph{residual graph} after edges
$u_1v_1,\ldots,u_kv_k$ are played in the matching game on $G$.

\begin{theorem}\label{general_lower}
$\gmmax(G) \ge \FR23\match(G)$ for every graph $G$.
\end{theorem}
\begin{proof}
As long as an edge remains, Max plays an edge belonging to a maximum matching;
this reduces the matching number by $1$.  When $\ap(G)\ge3$, the edge played by
Min in response is incident to at most two edges of a maximum matching and hence
reduces the matching number (of the residual graph) by at most $2$.  Hence a
round reduces $\ap$ by at most $3$ while adding $2$ to the number of edges
played.  When $\ap(G)=2$ and Max starts, two more edges will be played.
\end{proof}

Before proving that Theorem~\ref{general_lower} is sharp, we pause to show
that Min has a similar strategy using a small maximal matching to place an
upper bound on the outcome.

\begin{theorem}\label{general_upper}
$\gmmax(G)-1\le\gmmin(G) \le \FR32\almin(G)$ for every graph $G$.
\end{theorem}
\begin{proof}
Let $T$ be the set of $2\almin(G)$ vertices covered by a smallest maximal
matching $M$.  Min plays in $M$ when possible, using two vertices of $T$.
Max plays some edge, which uses at least one vertex of $T$.
This continues for $k$ rounds, where $k\ge\almin(G)/2$, using at least
$3k$ vertices of $T$ and making $2k$ moves, after which no edges remain in $M$.
Subsequently, at most $2\almin(G)-3k$ moves remain, since each move uses a
vertex of $T$.  Hence the total number of moves played is at most
$2\almin(G)-k$, which is at most $\FR32\almin(G)$.
\end{proof}

Sharpness of Theorem~\ref{general_lower} is shown by $rP_4$ with $r$ even;
Max can guarantee that at least half of the components contribute two edges.
Next we consider sharpness of Theorem~\ref{general_lower}.  A {\em split graph}
is a graph whose vertex set can be partitioned into a clique and an independent
set.  We present a Min strategy for the matching game on split graphs.  On many
split graphs, this strategy achieves equality in Theorem~\ref{general_lower}.

\begin{proposition}\label{split}
Let $G$ be a split graph.  If $V(G) = S \cup T$, with $S$ an independent set
and $T$ a clique, then $\gmmax(G) \le \ceil{\FR23\size{T}}$.  
\end{proposition}
\begin{proof}
On each turn, Min plays an edge joining two vertices of $T$ if possible, and
any legal move otherwise.
By the choice of $S$ and $T$, every edge in $G$ has at least one endpoint in
$T$.  Thus each move by Max covers at least one vertex of $T$, and each move
by Min covers two vertices of $T$ while two remain.  Thus each round increases
the size of the matching by $2$ and decreases $|T|$ by at least $3$, until
at most two vertices remain in $T$; the small cases are then checked explicitly.
\end{proof}

If a clique $T$ contains an endpoint of every edge in $G$, then
$\match(G) \le \size{T}$.  If $\ap(G)=|T|$ and $\size{T}\equiv 0\mod 3$, then
the lower bound in Proposition~\ref{split} matches the upper bound in
Theorem~\ref{general_lower}.  Thus equality holds in Theorem~\ref{general_lower}
for such split graphs.  We can introduce all the edges joining the
clique and the independent set, as in the next example.

\begin{example}\label{general_lower_tight}
For $k\ge1$ and $n\ge6k$, form $G$ from $K_n$ by deleting the edges of a
complete subgraph with $n-3k$ vertices; $G$ is a split graph whose clique $T$
has $3k$ vertices.  Note that $\match(G)=3k$ and $\Max(G)=2k$.  In fact, also
$G$ is $3k$-connected.
\qed
\end{example} 

The graph in Example~\ref{general_lower_tight} has the most edges among
$n$-vertex graphs such that $\ap(G)=3k$ and $\Max(G)=2k$.  To prove this, we
need several lemmas about ordinary matching.  The first is a special case of
a more difficult result of Brandt~\cite{Bra}.  The special case has a short,
self-contained proof.  Let $\delta(G)$ denote the minimum vertex degree in $G$.

\begin{lemma}\label{large_matching}
If $G$ is an $n$-vertex graph, then $\ap(G)\ge\min\{\floor{n/2},\delta(G)\}$.
\end{lemma}
\begin{proof}
Let $M$ be a maximum matching in $G$.  Suppose that $|M|<\floor{n/2}$, and let
$u$ and $v$ be distinct vertices not covered by $M$.  By the maximality of $M$,
all neighbors of $u$ or $v$ are covered by $M$.  If $d(u)+d(v)>2|M|$, then at
least three edges join some edge $xy\in M$ to $\{u,v\}$.  Now $xy$ can be
replaced with two edges from $\{x,y\}$ to $\{u,v\}$ to form a bigger matching.
Hence $d(u)+d(v)\le 2|M|$, and thus $\ap(G)\ge\min\{d(u),d(v)\}\ge\delta(G)$.
\end{proof}

The next result was observed by Plummer~\cite{P}.

\begin{lemma}\label{plummer}
If $G$ has $n$ vertices, and $\delta(G)\ge\FL{n/2}+1$, then every edge of
$G$ lies in a matching of size $\FL{n/2}$.
\end{lemma}
\begin{proof}
For $uv\in E(G)$, let $G'=G-\{u,v\}$.  Since $\delta(G')\ge\FL{(n-2)/2}$,
Lemma~\ref{large_matching} implies $\ap(G')\ge \FL{(n-2)/2}$.  Replace $uv$.
\end{proof}

\begin{lemma}\label{vertex_matched}
If $v$ is a non-isolated vertex of a graph $G$, then some maximum matching in
$G$ contains an edge incident to $v$.
\end{lemma}
\begin{proof}
When $uv$ is an edge, a maximal matching must cover $u$ or $v$.  If it covers
$u$ and not $v$, then the edge covering $u$ can be replaced with $uv$.
\end{proof}

\begin{theorem}
If $G$ is an $n$-vertex graph with $\match(G) = 3k$ and $\gmmax(G) = 2k$,
then $\size{E(G)} \le \binom{3k}{2} + 3k(n-3k)$.
\end{theorem}
\begin{proof}
We use induction on $k$.  For $k = 0$, the claim is $\size{E(G)} \le 0$, as
required by $\ap(G)=0$.
\looseness-1

For $k\ge1$, we seek an edge $uv$ for Max to play such that
$\ap(G-\{u,v\})=\ap(G)-1$ and $d(u)+d(v)\le n-1+3k$; call this a {\it good
edge}.  Playing a good edge eliminates at most $n-2+3k$ edges from the residual
graph.  With $n-2$ vertices remaining, the next move by Min eliminates at most
$2n-7$ more edges.

Let $G'$ be the residual graph after Max plays a good edge $uv$ and Min then
plays an optimal move.  By Prop\-osition~\ref{optimal},
$\gmmax(G) \ge 2 + \gmmax(G')$, so $\gmmax(G') \le 2k-2$.  Playing $uv$ reduces
the matching number only by $1$, and the move by Min reduces it by at most $2$,
so $\match(G') \ge 3k-3$.  Now Theorem~\ref{general_lower} yields
$\gmmax(G') = 2k-2$ and $\match(G') = 3k-3$.  Thus the induction hypothesis
applies to $G'$.  Adding the edges of $G$ not in $G'$ yields
\begin{align*}
\size{E(G)} &\le \size{E(G')} + 3n+3k-9
        \le \binom{3(k\!-\!1)}{2} + 3(k\!-\!1)(n\!-\!4-3(k\!-\!1)) + 3n+3k-9\\
              &= \frac{(3k-3)(3k-4)}{2} + 9k-6+3k(n-3k)
              = \binom{3k}{2} + 3k(n-3k).
\end{align*}

In most cases, $G$ has a good edge.  We may assume that $G$ has no isolated
vertices, since discarding them does not affect the matching number, the game
matching number, or the number of edges, and the edge bound for the smaller
graph is less than what we allow for $G$.  If $\delta(G)\le 3k$, then let $v$
be a vertex of minimum degree.  Lemma~\ref{vertex_matched} implies that some
maximum matching contains an edge $uv$ incident to $v$, and $d(v)\le 3k$
implies $d(u)+d(v)\le n-1+3k$.  Hence $uv$ is good.

If $n\ge 6k+2$, then $\FL{n/2}>3k$, and Lemma~\ref{large_matching} implies
$\delta(G) \le 3k$.  Since $\ap(G)=3k$ implies $n\ge6k$, we may henceforth
assume $n\in\{6k,6k+1\}$ and $\delta(G)\ge 3k+1=\FL{n/2}+1$.

Consider first the case $\delta(G)\ge\FL{n/2}+2$, and let $G'$ be the residual
graph after the first move by Max.  Each remaining vertex loses at most two
incident edges, so $\delta(G')\ge\FL{(n-2)/2}+1$.  Lemma~\ref{plummer} then
implies that every edge of $G'$ lies in a matching of size $\FL{(n-2)/2}$.
Thus even after the subsequent move by Min, the matching number is still at
least $\FL{n/2}-2$.  In other words, the first two moves of the game produce a
residual graph $G''$ with $\ap(G'')\ge 3k-2$.  By Theorem~\ref{general_lower},
$\Max(G'')>2k-2$, and hence $\Max(G)>2k$.

The remaining case is $\delta(G)=3k+1=\FL{n/2}+1$.  Every edge lies in a 
matching of size $3k$ (by Lemma~\ref{plummer}), so any edge with degree-sum at
most $n-1+3k$ is a good edge.  For a vertex $v$ of minimum degree, we conclude
that every neighbor $u$ of $v$ has degree at least $n-1$.  Since $n-1\ge3k+2$
when $n\ge6k\ge6$, we have $\delta(G-\{u,v\})\ge3k=\FL{(n-2)/2}+1$, and again
Min cannot reduce the matching number by $2$ after Max plays $uv$.
\end{proof}

\section{Forests}\label{forests}

In this section, we study the matching game on forests.  Although the lower
bound in Theorem~\ref{general_lower} is sharp infinitely often, we improve it
to $\Max(F)\ge\FR34\ap(F)$ when $F$ is a forest.  We also prove the natural
property that $\gmmin(F)\le \gmmax(F)$ when $F$ is a forest, which we apply in
the next section.

\begin{theorem}\label{forestbound}
If $F$ is a forest, then $\Max(F)\ge\FR34\ap(F)$.
\end{theorem}
\begin{proof}
Let $m=\ap(F)$.  Since the matching number of the residual graph is $0$ at
the end of the game, the number of moves played is $m-k$, where $k$ is the
number of moves on which the matching number of the residual forest declines by
$2$.  Such moves can only occur when the residual forest has a nonstar
component.

Suppose that such a component still exists at the beginning of a given round.
Let $x$ be an endpoint of a longest path $P$, with $P$ starting $x,w,v,\ldots$.
Max plays an edge $vu$ in a maximum matching of $F-wv$ (see
Lemma~\ref{vertex_matched}).  This move reduces $\ap$ by 1, as does the later
move in the remaining star at $w$ (whenever played).

The response by Min reduces $\ap$ by at most $2$.  Such moves by Min can
occur only after moves by Max that guarantee two good moves.  In following
the move-by-move sequence of values of $\ap$ on the residual graph, we
associate four units from $\ap(G)$ with each such reduction by $2$: the 
two lost on that move by Min, the one lost on the preceding move by Max, and
the one lost when an edge is later played in the resulting star left at $w$.
Hence $m\ge 4k$, and at least $3m/4$ moves are played.
\end{proof}

We next obtain a sufficient condition for sharpness in 
Theorem~\ref{forestbound}.

\begin{theorem}\label{forestsharp}
Let $n$ be a multiple of $8$.  If an $n$-vertex forest $F$ has $n/2$ leaves,
and the leaves are covered by $n/4$ disjoint copies of $P_4$, then
$\Max(F)=\FR34\ap(F)=\FR38 n$.
\end{theorem}
\begin{proof}
By the covering condition, no two leaves have a common neighbor, so the pendant
edges form a perfect matching, and then the upper bound follows from
Theorem~\ref{forestbound}.

We provide a strategy for Min.  In each of the first $n/8$ rounds, Min ensures
that two vertices become isolated (at most two untouched copies among the
$n/4$ special copies of $P_4$ will be used).  Isolating $n/4$ vertices ensures
that at most $\FR38 n$ edges will be played.

If Max plays an edge in one of the copies of $P_4$, then Min plays the central
edge in another copy of $P_4$.  If Max plays an edge joining two copies of
$P_4$ (they need not still be intact), then this move already isolates two
vertices.  Min then plays an edge at a leaf in one of these copies, if such an
edge exists, or otherwise plays any edge in an untouched copy of $P_4$.
\end{proof}

Forests satisfying the conditions of Theorem~\ref{forestsharp} are obtained
from $rP_4$ by adding edges joining components.  For example, one may add a
pendant edge at every vertex of $P_{4k}$.

The statement $\gmmin(G)\le \gmmax(G)$ essentially means that there is no
advantage to playing second rather than first.  When this holds in a hereditary
family, there is also no advantage to skipping a turn.  We prove this for
the family of forests.  As in Theorem~\ref{monotone}, the inductive proof is 
facilitated by simultaneously proving another claim, which is that adding
star components does not affect the game on the original forest $F$; neither
player can gain by playing on the added star instead of $F$. 

\begin{theorem}\label{forest_ineq}\label{forest_plus_star}
For every forest $F$ and all $t\in\NN$, the following two statements hold:\\
(1) $\Min(F)\le\Max(F)$.\\
(2) $\Min(F+K_{1,t})=1+\Min(F)$ and $\Max(F+K_{1,t})=1+\Max(F)$.
\end{theorem}
\begin{proof}
We use induction on $|V(F)|$.  Both statements hold by inspection for
$|V(F)|\le 2$, so consider larger $F$.  Suppose that (1) and (2) hold for
forests smaller than $F$.

{\bf Step 1:} {\it (1) holds for $F$.}
Let $k=\Min(F)$.  By Theorem~\ref{at_most_one}, the claim holds unless
$\Max(F)=k-1$.   Since $\Min(F)=\Max(F)$ when every component is a star, we may
choose a non-star component $C$ in $F$.  Since $C$ is not a star, a longest
path in $C$ has at least four vertices; let the first three be $u,v,w$ in order
($u$ is a leaf).

Let $F'=F-\{v,w\}$.  By Proposition~\ref{optimal}, $\Min(F)\le 1+\Max(F')$, so
$\Max(F') \ge k-1$.  Similarly, $\Max(F)\ge 1+\Min(F')$, so $\Min(F')\le k-2$.
Now Theorem~\ref{at_most_one} requires $\Max(F')=k-1$ and $\Min(F')=k-2$.
By Corollary~\ref{del_vertex}, also $\Max(F-w) = k-1$.

Obtain $F^*$ from $F$ by deleting $v$ and all its neighbors; note that
$F-w=F^*+K_{1,t-1}$, where $t = d(v)$.  Moreover, $F^*$ differs from $F'$ only
by deleting isolated vertices, so $\Max(F^*)=\Max(F')$.  Applying (2) to the
smaller forest $F^*$ now yields the contradiction
$$k-1=\Max(F-w)=\Max(F^*+K_{1,t-1})=1+\Max(F^*)=1+\Max(F')=k.$$

{\bf Step 2:} {\it (2) holds for $F$.}
We prove the four needed inequalities by considering an optimal first move $uv$
in the Min-start or Max-start game on $F$ or on $F+K_{1,t}$.  In the displayed
computations when $uv$ is chosen from $F+K_{1,t}$, we use the choice of $uv$,
the validity of (2) for $F'$, and Proposition~\ref{optimal}, in that order.
When $uv$ is an optimal first move on $F$, we use the same three facts in the
reverse order.  In each case, let $F'=F-\{u,v\}$ when $uv\in E(F)$.

{\bf 2a:} 
{\it $uv$ is an optimal first move in the Min-start game on $F+K_{1,t}$.}
If $uv\notin E(F)$, then $\gmmin(F+K_{1,t}) = 1+\gmmax(F) \ge 1+\gmmin(F)$,
since (1) holds for $F$ (by Step 1).  Otherwise,
$$\Min(F+K_{1,t}) = 1+\Max(F'+K_{1,t}) = 1+\Max(F')+1 \ge 1+\Min(F).$$

{\bf 2b:}
{\it $uv$ is an optimal first move in the Min-start game on $F$:}
$$\Min(F+K_{1,t}) \le 1+\Max(F'+K_{1,t}) = 1+\Max(F')+1 = 1+\Min(F).$$

{\bf 2c:}
{\it $uv$ is an optimal first move in the Max-start game on $F$:}
$$\Max(F+K_{1,t}) \ge 1+\Min(F'+K_{1,t}) = 1+\Min(F')+1 = 1+\Max(F).$$

{\bf 2d:} 
{\it $uv$ is an optimal first move in the Max-start game on $F+K_{1,t}$.}
If $uv\notin E(F)$, then $\Max(F+K_{1,t}) = 1+\Min(F) \le 1+\Max(F)$,
since (1) holds for $F$ (by Step 1).  Otherwise,
$$\Max(F+K_{1,t}) = 1+\Min(F'+K_{1,t}) = 1+\Min(F')+1 \le 1+\Max(F).$$

\vspace{-2.4pc}
\end{proof}

\section{Graphs with Small Maximum Degree}\label{smalldeg}

The following corollary of Theorem~\ref{forest_plus_star} is sometimes useful.

\begin{corollary}\label{forest_plus_star_corr}
For a forest $F$ and an integer $t$, an optimal move by the first player in
the game on $F$ is also optimal for that player on $F+K_{1,t}$.
\end{corollary}
\begin{proof}
Consider the Max-start game; the proof for the Min-start game is analogous.
Let $uv$ be an optimal first move in the Max-start game on $F$, and let
$F'=F-\{u,v\}$.  Using Proposition~\ref{optimal} and
Theorem~\ref{forest_plus_star} twice each yields
$$1+\Min(F')=\Max(F)=\Max(F+K_{1,t})-1\ge \Min(F'+K_{1,t})=1+\Min(F').$$
Hence $\Max(F+K_{1,t})=1+\Min(F'+K_{1,t})$, making $uv$ an optimal move for
Max on $F+K_{1,t}$.
\end{proof}

As an application of Corollary~\ref{forest_plus_star_corr}, we determine the
asymptotic value of $\gmmax(P_n)$.  While the corollary is not strictly needed
to prove this result, it simplifies the argument.

\begin{theorem}\label{gmmax_path}
For all $n$, we have $3\FL{\FR{n}{7}}\le \Max(P_n) \le 3\ceil{\frac{n}{7}}$.
\end{theorem}
\begin{proof}
For $7k\le n\le 7k+6$, the claimed statement is $3k\le \Max(P_n)\le 3(k+1)$.
Hence it suffices to prove the claim when $7\mid n$ and apply monotonicity from
Theorem~\ref{monotone}(2).

At each point during a game on $P_{7k}$, the residual graph is a disjoint union
of paths.  Each move increases the number of components in the residual graph
by at most $1$. 

{\it Upper bound:}  We give a strategy for Min.  Min always plays the second
edge of a longest remaining path, unless only isolated vertices and edges
remain.  Let $t$ be the number of turns played when the second phase begins,
with $s_1$ isolated vertices and $s_2$ isolated edges.
There are at most $t+1$ components at that time, so $s_1 + s_2 \le t+1$.  Since
each move deletes two vertices, $s_1 + 2s_2 + 2t = 7k$.  Now
$$7k + s_1 = 2t + 2s_2 + 2s_1 \le 2t + 2(t+1) = 4t+2.$$

Since each move by Min in the first phase isolates at least one vertex, 
$$(t-1)/2 \le s_1 \le 4t+2 - 7k.$$
Simplifying yields $t\ge\ceil{(14k-5)/7}=2k$.  Moreover,
$s_1\ge\ceil{(t-1)/2}\ge k$.  Since exactly $s_1$ vertices remain unmatched at
the end of the game, $\gmmax(P_{7k}) \le (n - s_1)/2 \le 3k$.

{\it Lower bound}:  We give a strategy for Max.  Max always plays the third
edge of a longest remaining path, unless no remaining path has four vertices,
in which case Max plays any edge of a longest remaining path.  By
Corollary~\ref{forest_plus_star_corr}, we may assume that Min plays an
isolated edge only when no other moves remain.

Let $t$ be the number of turns played when the second phase begins, with $s_i$
remaining components having $i$ vertices, for $i\le 3$.  Since each previous
move by Max created an isolated edge, and no isolated edges have yet been 
played, $s_2 \ge t/2$.  Each move has increased the number of components by at
most 1, so $s_1+s_2+s_3\le t+1$.  Thus $s_1+s_3 \le t/2 + 1$.

Counting the vertices played and the vertices remaining yields
$$7k = 2t+s_1+2s_2+3s_3 \ge 2t+(s_1+s_3)+2s_2 \ge 3t+(s_1+s_3),$$
so $s_1+s_3\le 7k-3t$.  The sum of $1/7$ times this inequality and $6/7$ times
$s_1+s_3\le t/2+1$ is $s_1+s_3\le k+6/7$.  By integrality, $s_1+s_3\le k$.
At the end of the game, exactly $s_1+s_3$ vertices remain unmatched, so
$\Max(P_{7k})\ge3k$.
\end{proof}

Our final results concern regular graphs.  When $G$ is $3$-regular,
$\match(G)\ge\lceil{4(\size{V(G)}-1)/9}\rceil$, and this is sharp~\cite{BDDFK}.
We seek an analogous sharp lower bound for $\Max(G)$ in terms of $\size{V(G)}$.
Using the bound for $\match(G)$, Theorem~\ref{general_lower} yields
approximately $\match(G) \ge \FR8{27}\size{V(G)}$.  An easy argument proves a
stronger bound, which we phrase for general regular graphs.

\begin{proposition}\label{cubic_lower_easy}
If $G$ is a connected $n$-vertex $r$-regular graph, then
$\gmmax(G)\ge\FR{rn-2}{4r-3}$.
\end{proposition}
\begin{proof}
We give a strategy for Max to ensure that edges are removed from the residual
graph ``slowly''.  Max first plays any edge.  On each turn thereafter, Max
plays any edge incident to a vertex of smallest nonzero degree.

Since $G$ is connected, the residual graph at any time after the first move
has a vertex of nonzero degree less than $r$.  Thus each move by Max after the
first deletes at most $2r-2$ edges.  Each move by Min deletes at most $2r-1$
edges.  Since Max moves before Min, after $k$ turns the residual graph contains
at least $\FR r2n - \floor{\FR{4r-3}2 k} - 1$ edges.  To end the game, all edges
must be deleted.  At that time $k\ge \FR{rn-2}{4r-3}$.
\end{proof}

For $r\ge 5$, the bound from Proposition~\ref{cubic_lower_easy} is weaker than
the bound obtained from Theorem~\ref{general_lower} by using the best-known
lower bounds on $\match(G)$.  Henning and Yeo~\cite{HY} proved
$\ap(G)\ge \frac{(r^3-r^2-2)n-2r+2}{2r^3-6r}$ for odd $r$, and this is sharp.
Multiplied by $2/3$ from Theorem~\ref{general_lower}, the lower bound on
$\Max(G)$ would be asymptotic to $n/3$ for large $r$ and $n$, while the lower
bound from Proposition~\ref{cubic_lower_easy} is only about $n/4$.

However, for $r=3$ Proposition~\ref{cubic_lower_easy} yields
$\Max(G)\ge \FR n3-\FR29$.  It is easy but tedious to improve the lower bound
to $n/3$ by considering the end of the game.  If Min moves last and still
eliminates five remaining edges, then the play by Min is the center of a double
star, and the previous move by Max could not have reduced all four leaves to
degree $1$.  Hence a vertex of degree $1$ was available to Max, and Max would
have deleted only three edges instead of four.  This implies that the last two
moves delete at most eight edges.  Case analysis along these lines improves the
additive constant to $0$.  We omit the details, partly since we do not believe
that $1/3$ is the best coefficient on the linear term.

\begin{example}\label{disconn}
For $3$-regular graphs not required to be connected, there is an $n$-vertex
graph with $\Max(G)=\Min(G)=3n/8$ when $16\mid n$.  Each component $H$ has a
central vertex $x$ whose deletion leaves $3K$, where $K$ is the $5$-vertex graph
obtained by subdividing one edge of $K_4$.

A short case analysis shows that $\Max(H)=\Min(H)=6$, with four vertices 
isolated at the end of the game.  Since $\Max(H)=\Min(H)$ and the value is
even, each player can respond to the other in the same component and guarantee
never doing worse than $3n/8$, component by component.
\qed
\end{example}

Within the smaller family of connected cubic graphs, our construction is weaker.

\begin{theorem}\label{cubic_upper}
There is a sequence $G_1, G_2,\ldots$ of connected cubic graphs such that
$\gmmax(G_k) < (7/18)\size{V(G_k)}$.
\end{theorem}
\begin{proof}
Again let $K$ be the $5$-vertex graph formed by subdividing one edge of $K_4$.
Let $T_k$ denote the {\em complete cubic tree of height $k$}, the rooted tree
in which each non-leaf vertex has degree 3 and each leaf has distance $k$ from
the root.  Let $G_k$ be the $3$-regular graph formed by identifying each leaf
in $T_{k+1}$ with the vertex of degree $2$ in a copy of $K$.  Note that $G_0$
is the graph in Example~\ref{disconn}.

When every non-leaf vertex in a tree has degree $3$, the number of leaves is
the number of non-leaves plus $2$.  Since $T_{k+1}$ has $3\cdot 2^k$ leaves, it
thus has $3\cdot 2^k-2$ non-leaves, and $G_k$ has $18\cdot 2^k-2$ vertices.

We give a strategy for Min in the game on $G_k$.  Let $B$ be the $6$-vertex
graph consisting of $K$ plus a pendant edge at the vertex having degree $2$ in
$K$.  In $G_k$ there are $3 \cdot2^k$ copies of $B$, each containing one leaf
of $T_{k+1}$ as its cut-vertex.  Let $F$ be the set of three edges in $B$ that
lie in no perfect matching in $B$.

If a move by Max is the first edge played in some copy of $B$, then Min plays
an edge of $F$ in that copy of $B$ (at least one such edge is available).
Otherwise, Min plays an edge of $F$ in some other copy of $B$, if possible.
If already some edge has been played in every copy of $B$, then Min plays any
legal move.

Min thus ensures that at most two edges are played in each copy of $B$.
To bound the number of edges played in $T_{k+1}$ that are not in the copies of
$B$, we bound $\match(T_k)$.  Let $S$ be the set of vertices in $T_k$ whose
distance from a nearest leaf is odd.  Since $T_k-S$ has no edges, 
$\ap(T_k)\le |S|$.  Since the number of vertices at distance $i$ from the root
is at most half the number of vertices at distance $i+1$, we have
$\size{S} \le (1/3) \size{V(T_k)}$.  Thus
$\match(T_k)\le \FL{(1/3)(6\cdot 2^{k-1}-2)}= 2^k-1$.
It now follows that
$$
\Max(G_k)\le 2^k-1 + 6\cdot2^k = 7\cdot2^k-1 < \frac{7}{18}(18\cdot2^k-2).
$$

\vspace{-2.4pc}
\end{proof}

Note that in Theorem~\ref{cubic_upper} we did not determine $\Max(G_k)$.
In fact, Max can force two edges in each copy of $B$ and can force a 
matching that covers $S$ in $T_k$, so the bound is essentially sharp. 
We omit these arguments, because the relevant question here is the upper bound.

\end{document}